\documentclass{amsart}
\usepackage{amssymb,latexsym}
\theoremstyle{plain}
\newtheorem{theorem}{Theorem}
\newtheorem{corollary}{Corollary}

\newtheorem{lemma}{Lemma}
\theoremstyle{definition}
\newtheorem{definition}{Definition}
\newtheorem{example}{Example}

\newtheorem{remark}{Remark}
\newtheorem{question}{Question}

\def\PP{\mathbb{P}^}
\def\sr{\mathrm{sr}}
\def\sbr{\mathrm{sbr}}

\date{}

\begin{document}

\title[Decomposition of polynomials]
{Decomposition of homogeneous polynomials with low rank}
\author{Edoardo Ballico, Alessandra Bernardi}
\address{Dept. of Mathematics\\
  University of Trento\\
38123 Povo (TN), Italy}
\address{GALAAD, INRIA M\'editerran\'ee,  BP 93, 06902 Sophia 
Antipolis, France.}
\email{ballico@science.unitn.it, alessandra.bernardi@inria.fr}
\thanks{The authors were partially supported by CIRM of FBK Trento 
(Italy), Project Galaad of INRIA Sophia Antipolis M\'editerran\'ee 
(France), Institut Mittag-Leffler (Sweden), Marie Curie: Promoting science (FP7-PEOPLE-2009-IEF), MIUR and GNSAGA of 
INdAM (Italy).}

\maketitle
\begin{quote}ABSTRACT: \emph{Let $F$ be a homogeneous polynomial of 
degree $d$ in $m+1$ variables defined over an algebraically closed 
field of characteristic 0 and suppose that $F$ belongs to the 
$s$-th secant variety of the $d$-uple Veronese embedding
of $\mathbb{P}^m$ into  $ \PP {{m+d\choose d}-1}$ but that its 
minimal decomposition as a sum of $d$-th powers of linear forms
$M_1, \ldots , M_r$ is $F=M_1^d+\cdots + M_r^d$ with $r>s$. We show 
that if $s+r\leq 2d+1$ then such a decomposition of $F$
can be split in two parts:  one of them is made by linear forms that 
can be written using only two variables, the other part
is  uniquely determined once one has fixed the first part. We also 
obtain a uniqueness theorem for the minimal decomposition
of $F$ if $r$ is at most $d$ and a mild condition is satisfied.}
\end{quote}

\section*{Introduction}
The decomposition of a homogeneous polynomial that combines a minimum 
number of terms and that involves a minimum number of variables is a 
problem arising from classical Algebraic Geometry (\cite{ah}, \cite{ik}), Computational Complexity (\cite{ls}) 
and Signal Processing (\cite{vw}).
Any statement on homogeneous polynomials can be translated in an 
equivalent statement on symmetric tensors. In fact, if we indicate 
with $V$ a vector space of dimension $m+1$ defined over
a field $K$
of characteristic 0,
and with $V^*$ its dual space, then, for any positive integer $d$, 
there is an obvious identification between the vector space of 
symmetric tensors  $S^dV^*\subset (V^{*})^{\otimes d}$ and the space 
of homogeneous polynomials $K[x_0, \ldots , x_m]_d$ of degree $d$ 
defined over $K$.
In this paper we will always work with an algebraically closed field 
$K$ of characteristic $0$. The requirement that a form (or a 
symmetric tensor) involves a minimum number of terms is a quite 
recent and very interesting problem coming from applications.
Given a form $F\in K[x_0, \ldots , x_m]_d$ (or a symmetric tensor $T\in S^dV^*$),  the 
minimum positive integer $r$ for which there exist linear forms $L_1, 
\ldots , L_r\in K[x_0, \ldots , x_m]_1$ (vectors $v_1, \ldots , 
v_r\in V^*$ respectively) such that
\begin{equation}\label{FT}
F=L_1^d+ \cdots + L_r^d, \; \; \; \;   (T=v_1^{\otimes d}+ \cdots 
+v_r^{\otimes d})
\end{equation}
is called the {\it symmetric rank} $\sr(F)$ of $F$ ($\sr(T)$ of $T$ respectively).
Computations of the symmetric rank for a given form (or a given symmetric tensor) 
are studied in \cite{cs}, \cite{bgi}, \cite{bcmt} and \cite{bb}.
First of all we focus our attention on those particular 
decompositions of a form $F\in K[x_0, \ldots ,x_m]_d$ (or $T\in S^dV^*$) of the type 
(\ref{FT}) with $r=\sr(F)$ ($r=\sr(T)$ respectively). What about the possible 
uniqueness of the decomposition of such a form $F$ ($T$ respectively)?
A general form, for example, can have a unique decomposition as in 
(\ref{FT}) only if $\frac{1}{n+1}{n+d \choose n}\in \mathbb{Z}$ (see 
\cite{rs}, \cite{me}, \cite{me1}, \cite{cmr} also for further results 
on this normal form). If the polynomial is not general, very few 
things are known. 

Let $X_{m,d}\subset \mathbb{P}^N$, with $m \ge 1$,  $d \ge 2$ and 
$N:= {{m+d}\choose m}-1$, be the classical Veronese variety obtained 
as the image of the $d$-uple Veronese embedding $\nu _d: \PP m \to \PP 
N$. The {\it $s$-th secant variety} $\sigma_s(X_{m,d})$ of the 
Veronese variety $X_{m,d}$ is the Zariski closure in $\mathbb {P}^N$
of the union of all linear spans $\langle P_1, \ldots , P_s \rangle$ 
with $P_1, \ldots , P_s \in X_{m,d}$.
   For  any point $P\in \PP N$, we indicate with $\sbr(P)=s$ the 
minimum integer $s$ such that $P\in \sigma_s(X_{m,d})$. This integer 
is called the {\it symmetric border rank} of $P$.
By a famous theorem of J. Alexander and A. Hirschowitz all integers 
$\dim (\sigma _s(X_{m,d}))$ are known (\cite{ah},  \cite{ci}, 
\cite{bo}).  Since  $\mathbb{P}^m\simeq \mathbb{P}(K[x_0, \ldots , 
x_m]_1)\simeq \mathbb{P}(V^*)$, the generic element belonging to 
$\sigma_s(X_{m,d})$ is  the projective class of a form (a symmetric tensor) 
of type (\ref{FT}).
Unfortunately, for a given $P\in \mathbb{P}^N$, we only have the inequality $\sbr(P) 
\leq  \sr(P)$. For the forms $F$ for which the decomposition 
(\ref{FT}) is not unique, it makes sense to study those different 
decompositions.
There is a uniqueness
theorem for general points with prescribed non-maximal symmetric border rank
$s$ using the notion of $(s-1)$-weakly non-defectivity introduced
by C. Ciliberto and L. Chiantini (\cite{cc}, \cite{cr},
Proposition 1.5).
In this paper we are interested in those particular decompositions of 
a given $F\in K[x_0, \ldots , x_m]_d$  of the type (\ref{FT}) with 
$r=\sr(F)$ and $\sbr(F)<\sr(F)$ ($T\in S^dV^*$ respectively).
In many applications one would like to reduce the number of 
variables,  at least for a part of the data.
For such a particular choice of $F$, is it possible to find linear 
forms $L_1, L_2,M_1, \ldots , M_t\in K[x_0, \ldots , x_m]_1$ and a 
binary form $Q\in K[L_1, L_2]_d$, such that a given polynomial $F\in 
K[x_0, \ldots , x_m]_d$ can be written as
$F=Q+M_1^d+\cdots + M_t^d$?
(On normal forms of homogeneous polynomials see also \cite{lt}, 
\cite{cm}, \cite{ik}.)
The main result of this paper is the following.

\begin{theorem}\label{a1}
Let $P\in \PP N$ with $N={m+d\choose d}-1$.
Suppose that:
$$\begin{array}{c}
\sbr(P)<\sr(P) \hbox{ and}
\\
\sbr(P)+\sr(P)\leq 2d+1.
\end{array}$$
Let $\mathcal{S}\subset X_{m,d}$ be a $0$-dimensional reduced
subscheme that realizes the symmetric rank of $P$, and let
$\mathcal{Z}\subset X_{m,d}$ be a smoothable $0$-dimensional non-reduced
subscheme such that $P\in \langle \mathcal{Z} \rangle$ and
$\deg{\mathcal{Z}}\leq \sbr(P)$. Let also $C_d\subset X_{m,d}$ be the unique
rational normal curve that intersects $\mathcal{S}\cup
\mathcal{Z}$ in degree at least $d+2$. Then, for all points $P\in \PP N$ as above we have that:
$$
\mathcal{S}=\mathcal{S}_1\sqcup \mathcal{S}_2,\ \ \ \ 
\mathcal{Z}=\mathcal{Z}_1\sqcup \mathcal{S}_2,
$$
where $\mathcal{S}_1=\mathcal{S}\cap C_d$,
$\mathcal{Z}_1=\mathcal{Z}\cap C_d$ and
$\mathcal{S}_2=(\mathcal{S}\cap \mathcal{Z})\setminus
\mathcal{S}_1$.
\\
Moreover $\deg (\mathcal {Z}) = \sbr (P)$ and the scheme 
$\mathcal{S}_2$ is unique.
\end{theorem}

The existence of such a  scheme $\mathcal{Z}$  was known from 
\cite{bgi} and \cite{bgl} (see Remark \ref{Z}). The assumption 
``~$\sbr(P)+\sr(P)\leq 2d+1$~''  is sharp (see Example \ref{00}).

In the language of polynomials, Theorem \ref{a1} can be rephrased as follows.

\begin{corollary}\label{corpoly} Let $F\in K[x_0, \ldots , x_m]_d$ be such that
$\sbr(F)+\sr(F)\leq 2d+1$ and $\sbr(F)<\sr(F)$. Then there are an integer $t \ge 0$, linear forms $L_1, L_2,M_1, \ldots , 
M_t\in K[x_0, \ldots , x_m]_1$,  and a form $Q\in K[L_1,L_2]_d$ such that 
$F=Q+M_1^d+\cdots + M_t^d$, $t \le \sbr(F)+\sr(F)-d-2$, and $\sr(F)=\sr(Q)+t$.
Moreover $t$, $M_1,\dots , M_t$ and the linear span of $L_1, L_2$ are 
uniquely determined by $F$. \end{corollary}

An analogous corollary can be stated for symmetric tensors.

\begin{corollary}\label{cortens} Let $T\in S^dV^*$ be such that
$\sbr(T)+\sr(T)\leq 2d+1$ and $\sbr(T)<\sr(T)$. Then there are an integer $t\ge 0$, vectors $v_1, v_2,w_1, \ldots , 
w_t\in S^1V^*$, and a symmetric tensor $v\in S^d(\langle v_1, v_2 \rangle)$ such that
$T=v+w_1^{\otimes d}+\cdots + w_t^{\otimes d}$, $t\le \sbr(T)+\sr(T)-d-2$, and $\sr(T)=\sr(v)+t$.
Moreover $t$, $w_1,\dots ,w_t$ and   $\langle v_1, v_2\rangle $ are 
uniquely determined by $T$. \end{corollary}

Observe that
the variables $L_1,L_2$ in Corollary \ref{corpoly} and the vectors 
$v_1,v_2$ in Corollary \ref{cortens} correspond to the line
$\ell\subset \PP m$ such that $C_d:= \nu _d(\ell)$ is the rational normal 
curve introduced in Theorem
\ref{a1}.  Moreover the integer $t$  in Corollaries \ref{corpoly} and 
\ref{cortens} is $\sharp (\mathcal {S}_2)$ where
$\mathcal{S}_2$ is as in Theorem \ref{a1}.
The decompositions $Q = R_1^d+\cdots +R_{r'}^d$ 
with $R_i\in K[L_1,L_2]_1$, are not unique (analogously the 
decompositions $v=u_1^{\otimes d} + \cdots + u_{r'}^{\otimes d}$ with 
$u_i\in \langle v_1, v_2 \rangle$), but one of them may be found using
Sylvester's algorithm or any of the available algorithms (\cite{cs}, 
\cite{lt}, \cite{bgi}). Unfortunately, given $F$ as in  Corollary 
\ref{corpoly} ($T$ as in Corollary \ref{cortens} respectively) we do not 
have any explicit algorithm to find $M_1,\dots ,M_t\in K[x_0, \ldots , x_m]_d$ and hence 
$Q\in K[L_1, L_2]_d$ ($w_1, \ldots , w_t\in S^1V^*$ and $v\in 
S^d(\langle v_1 , v_2 \rangle )$ respectively).

Using Theorem \ref{a1}  and a related lemma (Lemma \ref{dd}) it is 
also possible to address the question on the uniqueness of the 
decomposition (\ref{FT}).

\begin{theorem}\label{oo}
Assume $d\ge 5$. Fix a finite set $B\subset \mathbb {P}^m$ such
that $\rho:= \sharp (B) \le d$ and no subset of it with
cardinality $\lfloor (d+1)/2\rfloor$ is collinear. Fix $P\in
\langle \nu _d(B)\rangle$ such that $P\notin \langle \mathcal{E}\rangle$
for any $\mathcal{E}\subsetneqq \nu _d(B)$. Then $\sr (P)=\sbr (P) =\rho$
and $\nu _d(B)$ is the only 0-dimensional scheme $\mathcal {Z}
\subset X_{m,d}$ such that $\deg (\mathcal {Z})\le \rho$ and $P\in
\langle \mathcal {Z}\rangle$.
\end{theorem}

Unfortunately, for a given $P\in \mathbb {P}^N$ that satisfies the 
hypothesis of Theorem \ref{oo} we are not able to give 
explicitly the set $B$. Knowing the uniqueness of a decomposition is 
very interesting both from the applications  and the pure 
mathematical point of view, but
very few results are known.
Theorem \ref{oo} is an extension of \cite{bgl} with an additional
assumption. It is worth noting that without some additional 
assumption \cite{bgl}, Theorem 1.2.6, cannot be extended (e.g., it
is sharp when $m=1$). We give an example showing that if $m=2$,
then Theorem \ref{oo} is sharp (see Example \ref{o+}), even taking
$B$ in linearly general position.

\section{Preliminaries}\label{sec1}

In this section we prove two auxiliary lemmas that will be crucial in 
the proof of the main result of this paper. Theorems \ref{a1} and 
\ref{oo} are well-known if $m=1$ since Sylvester. Hence we may assume 
that $m\ge 2$.

\begin{definition}\label{def}
We  say that a smoothable 0-dimensional scheme $\mathcal{Z}\subset 
X_{m,d}$ computes  the symmetric border rank $\sbr(P)$ of $P\in \mathbb{P}^N$ if $\deg (\mathcal {Z}) = \sbr (P)$
and $P\in \langle \mathcal{Z} \rangle$.
A reduced 0-dimensional scheme
$\mathcal{S} \subset X_{m,d}$
computes the symmetric rank $\sr(P)$ of $P\in \mathbb{P}^N$ if $\sharp (\mathcal {S}) =\sr (P)$
and $P\in \langle \mathcal{S} \rangle$.
\end{definition}

By the definition of symmetric
rank, if $\mathcal {S}$ computes $\sr(P)$, then $P\notin \langle \mathcal{S}' 
\rangle$ for any reduced  0-dimensional scheme $\mathcal{S}' \subset 
X_{m,d}$ with $\deg(\mathcal{S}')< \deg(\mathcal{S})$. Hence $\mathcal {S}$ is linearly
independent.

\begin{lemma}\label{c0}
Fix any $P\in \PP r$ and two 0-dimensional subschemes $A$, $B$ of 
$\PP r$ such that $A \ne B$, $P\in \langle A\rangle$,
$P\in \langle B\rangle$, $P\notin \langle A'\rangle$ for any 
$A'\subsetneqq A$ and $P\notin \langle B'\rangle$ for any 
$B'\subsetneqq B$.
Then $h^1(\PP r,\mathcal {I}_{A\cup B}(1)) >0$.
\end{lemma}

\begin{proof}
Since $A$ and $B$ are 0-dimensional, 
$h^1(\PP r,\mathcal {I}_{A\cup B}(1)) \ge
\max \{h^1(\PP r,\mathcal {I}_A(1)),
h^1(\PP r,\mathcal {I}_B(1))\}$. Thus we may assume $
h^1(\PP r,\mathcal {I}_A(1))=
h^1(\PP r,\mathcal {I}_B(1))=0$, i.e. $\dim (\langle A\rangle )=\deg 
(A)-1$ and $\dim (\langle B\rangle )= \deg (B)-1$.
Set $D:= A\cap B$ (scheme-theoretic intersection). Thus $\deg (A\cup 
B) = \deg (A)+\deg (B)-\deg (D)$.
Since $D\subseteq A$ and $A$ is linearly independent, we have $\dim 
(\langle D\rangle ) =\deg (D)-1$.
Since $h^1(\mathbb {P}^r,\mathcal {I}_{A\cup B}(1))>0$ if and only if
$\dim (\langle A\cup B\rangle )\le \deg (A\cup B)-2$, we get 
$h^1(\mathbb {P}^r,\mathcal {I}_{A\cup B}(1))>0$
if and only if $\langle D\rangle \subsetneqq \langle A\rangle \cap 
\langle B\rangle$.
Since $A \ne B$, then $D\subsetneqq A$. Hence $P\notin \langle 
D\rangle$. Since $P\in  \langle A\rangle\cap  \langle B\rangle$, we 
are done.
\end{proof}

The next observation shows the existence of the scheme $\mathcal 
{Z}\subset X_{m,d}$ that   computes the symmetric border rank of a 
point $P\in \mathbb{P}^N$ that satisfies the conditions of
Theorem \ref{a1}.

\begin{remark}\label{Z}
Fix integers $m \ge 1$, $d \ge 2$ and $P\in \PP N$ such that $\sbr 
(P) \le d+1$. By \cite{bgl},
Lemma 2.1.5, or \cite{bgi}, Proposition 11, there is a 
smoothable 0-dimensional scheme $\mathcal{E} \subset X_{m,d}$ such that $\deg (\mathcal{E}) \le 
\sbr(P)$ and $P\in \langle \mathcal{E}\rangle$.
Moreover,
$\sbr (P)$ is the minimal of the degrees of any such smoothable scheme $\mathcal{E}$.
\end{remark}

In the statement of Theorem \ref{a1} we claimed  the existence of a 
unique rational normal curve $C_d\subset X_{m,d}$  such that $\deg ((\mathcal{S}\cup
\mathcal{Z})\cap C_d) \ge d+2$. This will be a consequence of the following lemma where 
the line $\ell\subset \mathbb{P}^m$ and the scheme $W\subset 
\mathbb{P}^m$ will be used in the proof of Theorem \ref{a1} with
$\nu_d(\ell)=C_d$, while as $\nu _d(W)$ we will take several different schemes associated to $\mathcal{S}\cup
\mathcal{Z}$.

\begin{lemma}\label{p1}
Fix an integer $x\ge 1$. Let $W \subset \PP m$,  $m
\ge 2$, be a 0-dimensional scheme of degree $\deg(W) \le 2x+1$
and such that $h^1(\PP m,\mathcal {I}_W(x)) >0$. Then there is a unique
line $\ell\subset \PP m$ such that $\deg(\ell\cap W)\ge x+2$ and
$$\deg(W\cap \ell)=x+1+h^1(\PP m,\mathcal {I}_W(x)).$$\end{lemma}

\begin{proof}
For the existence of the line $\ell\subset \PP m$ see \cite{bgi}, Lemma 34.
\\
Since $\deg(W) \le 2x+1$ and since the scheme-theoretic intersection 
of two different lines
has length at most one and $\deg (W) \le 2x+2$, there is no line 
$R\ne \ell$ such that $\deg (R\cap W) \ge x+2$. Thus $\ell$ is unique.
\\
We prove  the formula $\deg(W\cap \ell)=x+1+h^1(\mathcal {I}_W(x))$ by 
induction on $m$.

First assume $m=2$. In this case $\ell$ is a Cartier divisor of $\PP m$. 
Hence the residual scheme
$\mbox{Res}_\ell(W)$ of $W$ with respect to $\ell$ has degree 
$\deg(\mbox{Res}_\ell(W))=\deg(W) -\deg(W\cap \ell)$. The exact 
sequence that defines the residual scheme $\mbox{Res}_\ell(W)$ is:
\begin{equation}\label{eqa1}
0 \to \mathcal {I}_{\mbox{Res}_\ell(W)}(x-1) \to \mathcal {I}_W(x) \to 
\mathcal {I}_{W\cap \ell,\ell}(x) \to 0.
\end{equation}
  Since $\dim (\mbox{Res}_\ell(W))\le \dim (W) \le 0$ and $x-1\ge -2$, we 
have $h^2(\PP m,\mathcal {I}_{\mbox{Res}_\ell(W)}(x-1))=0$.
Since $\deg(W\cap \ell) \ge x+1$, we have $h^0(\ell,\mathcal {I}_{W\cap 
\ell}(x))=0$. Since $\deg(\mbox{Res}_\ell(W))
= \deg(W) -\deg(W\cap \ell)  \le x$, we obviously have $h^1(\PP m,\mathcal 
{I}_{\mbox{Res}_\ell(W)}(x-1))=0$ (this is also a particular case of 
\cite{bgi}, Lemma
34). Thus the cohomology exact sequence of (\ref{eqa1}) gives 
$h^1(\PP m,\mathcal {I}_W(x)) = \deg(W\cap \ell)-x-1 $. This proves the
lemma for $m=2$.

Now  assume $m\ge 3$ and that the result is true for $\mathbb {P}^{m-1}$.
Take a general hyperplane $H \subset \PP m$ containing $\ell$ and set 
$W':= W\cap \ell$. The inductive assumption gives $h^1(H,\mathcal 
{I}_{W'}(x)) = \deg(W'\cap \ell)-x-1 $. Since $\deg(\mbox{Res}_H(W)) \le 
x-1$, we get, as above, $h^1(\PP m,\mathcal {I}_{\mbox{Res}_H(W)}(x-1))= 0$.
Consider now the analogue exact sequence of (\ref{eqa1}) using $H$ 
instead of $\ell$:
$$0 \to \mathcal {I}_{\mbox{Res}_H(W)}(x-1) \to \mathcal {I}_W(x) \to 
\mathcal {I}_{W\cap H,H}(x) \to 0.$$
Since $W\cap \ell = W'\cap \ell$, we get, as above, that
  $h^1(\PP m,\mathcal {I}_W(x))= \deg(W\cap \ell)-x-1 $.
\end{proof}

\section{The proofs}\label{results}

In this section we prove Theorems \ref{a1} and  \ref{oo}.

\vspace{0.3cm}

\qquad {\emph {Proof of Theorem \ref{a1}.}} The existence of the 
smoothable scheme $\mathcal{Z}\subset X_{m,d}$ that computes $\sbr(P)$ is assured by  Remark \ref{Z}. Any such smoothable scheme
has degree $\sbr(P)$ (Remark \ref{Z}).
Let $S$ (resp. $Z$) be the only subset (resp. subscheme) of $\PP m$ such that
$\mathcal {S} = \nu _d(S)$ (resp. $\mathcal {Z} = \nu _d(Z)$). By hypothesis  $\sharp (S) = \sr (P)$ and $\deg (Z) = \sbr (P)$. Set $W:= S\cup 
Z$ and $\mathcal{W}:=\nu _d(W)$. We have $\deg (W) = \sr (P) + \sbr 
(P) \le 2d+1$. Let $\mathcal {T}$ be a minimal subscheme of $\mathcal {Z}$ such that
$P\in \langle \mathcal {T}\rangle$. Since $\deg (\mathcal {T}) \le \deg (\mathcal {Z}) < \deg (\mathcal {S})$,
we have $\mathcal {T} \ne \mathcal {S}$. Lemma \ref{c0}
applied to $r:=N$, $A:= \mathcal {T}$ and $B:= \mathcal {S}$ gives $h^1(\mathcal {I}_{\mathcal{T}\cup \mathcal {S}}(1))>0$.
Thus $h^1(\mathcal {I}_{\mathcal{W}}(1))>0$. Thus
$\dim (\langle \mathcal{W}\rangle )\le \mbox{deg}(\mathcal{W})-2$. Since
$\deg(\mathcal{W}) \le \deg(\mathcal{Z}) + \deg(\mathcal{S})= 
\sbr(P)+\sr(P) \le 2d+1$ and $h^1(\mathcal 
{I}_{\mathcal{W}}(1))=h^1(\mathbb {P}^m,\mathcal {I}_W(d))$, there 
is a unique line $\ell\subset \mathbb {P}^m$ whose image
$C_d:=\nu_{d}(\ell)$ in $X_{m,d}$ contains a subscheme of $\mathcal{W}$ with 
length at least $d+2$ (Lemma \ref{p1}). Since $C_d = \langle 
C_d\rangle \cap X_{m,d}$ (scheme-theoretic intersection), we
have $\mathcal {W}\cap C_d = \nu _d(W\cap \ell)$, $\mathcal {Z}\cap C_d 
= \nu _d(Z\cap \ell)$ and $\mathcal {S}\cap C_d = \nu _d(S\cap \ell)$.
\\

\quad (a)   Let $\mathcal{S}_1, \mathcal{S}_2\subset \mathcal{S}$ be 
as  defined in the statement
and set $\mathcal{S}_3:=\mathcal{S}\setminus (\mathcal{S}_1\cup 
\mathcal{S}_2)$. Let $S_3\subset \PP m$ be the
only subset such that $\mathcal {S}_3 = \nu _d(S_3)$.  Set $W':= 
W\setminus S_3$ and $\mathcal{W}':= \nu _d(W') =\mathcal{W}\setminus 
\mathcal{S}_3$. Notice that $W'$ is well-defined, because each point 
of $S_3$ is a connected component of the scheme $W$.

In this step we prove $S_3=\emptyset$, i.e. $\mathcal{S}_3 = \emptyset$.

Assume that this is not the case and that $\sharp (\mathcal{S}_3)>0$.
Lemma \ref{p1} gives $h^1(\PP m,\mathcal {I}_{W\cap \ell}(d)) =
  h^1(\PP N,\mathcal {I}_{\mathcal{W}}(1))$ and $h^0(\mathcal 
{I}_{\mathcal{W}}(1)) = h^0(\mathcal {I}_{C_d\cap
\mathcal{W}}(1)) -\deg (\mathcal{W})+\deg (\mathcal{W}\cap C_d)$. Hence we get
$$\dim (\langle \mathcal{W}\rangle )=\dim (\langle 
\mathcal{W}'\rangle )+\sharp (\mathcal{S}_3).$$
Now, by definition, we have that $\mathcal{S}\cap \mathcal{W}' = 
\mathcal{S}_1\cup \mathcal{S}_2$, $\mathcal{W} =\mathcal{W}'\sqcup 
\mathcal{S}_3$ and $\mathcal{Z}\cup \mathcal{S}_1\cup
\mathcal{S}_2= \mathcal{W}'$. Grassmann's formula gives $\dim( 
\langle \mathcal{W}'\rangle \cap \langle \mathcal{S}\rangle )
= \dim (\langle \mathcal{W}'\rangle )+\dim (\langle 
\mathcal{S}\rangle ) - \dim (\langle
\mathcal{W}'\cup \mathcal{S}\rangle ) = \dim (\langle 
\mathcal{S}\rangle ) -\sharp (\mathcal{S}_3)$.
Since $\mathcal {S}$ is linearly independent, we have $\dim (\langle \mathcal{S}_1 
\cup \mathcal{S}_2 \rangle)= \dim(\langle
\mathcal{S}
\rangle) -
\sharp (\mathcal{S}_3)$. Hence
$\dim( \langle \mathcal{S}_1\cup \mathcal{S}_2 \rangle)=\dim (\langle 
\mathcal{W}'\rangle \cap \langle \mathcal{S} \rangle)$;
since $\langle \mathcal{S}_1 \cup \mathcal{S}_2 \rangle \subseteq 
\langle \mathcal{W}'\rangle \cap \langle \mathcal{S} \rangle$
we get $\langle \mathcal{S}_1 \cup \mathcal{S}_2 \rangle = \langle 
\mathcal{W}' \rangle \cap \langle \mathcal{S}
\rangle$. Since $P\in \langle \mathcal{Z} \rangle \cap \langle 
\mathcal{S} \rangle \subseteq \langle \mathcal{W}' \rangle \cap
\langle \mathcal{S} \rangle = \langle \mathcal{S}_1 \cup 
\mathcal{S}_2 \rangle$, we get that $P\in \langle \mathcal{S}_1 \cup
\mathcal{S}_2\rangle$.
  Since we supposed that $\mathcal{S}\subset X_{m,d}$ is a set computing the 
symmetric rank of $P$, it is absurd that $P$ belongs to the span of a 
proper subset of $\mathcal{S}$, then necessarily 
$\sharp(\mathcal{S}_3)=0$, that is equivalent to the fact that 
$\mathcal{S}_3=\emptyset$.
  Thus in this step we have just proved $\mathcal {S} =\mathcal {S}_1\sqcup 
\mathcal {S}_2$.
\\
\\
In steps (b), (c) and (d) we will prove $\mathcal {Z} = (\mathcal 
{Z}\cap C_d)\sqcup \mathcal {S}_2$ in a very similar way (using $Z$ 
instead of $S$). In each of these
  steps we take a subscheme $W_2 \subset W$ such that $S\subset W_2$, 
$W_2\cap \ell = W\cap \ell$ and $W_2\cup Z = W$. Then we play with Lemma 
\ref{p1}. In steps (b) (resp. (c), resp. (d)) we call $W_2 = W''$ 
(resp. $W_2 = W_Q$, resp. $W_2=W_1$). Since $\deg (\nu _d(Z)) \le d+1$, the scheme
$\nu _d(Z)$ is linearly independent.
\\

\quad (b) Let $Z_4\subset \mathbb{P}^n$ be the union of the connected components of $Z$ 
which do not intersect $\ell\cup S_2$. Here we prove
$Z_4=\emptyset$. Set $W'':= W\setminus Z_4$. The scheme $W''$ is 
well-defined, because $Z_4$ is a union of some of the
connected components of $W$. Lemma 2 gives $\dim (\langle\nu _d(W)\rangle ) = 
\dim (\langle \nu _d(W'')\rangle ) + \deg (Z_4)$.
Since $W = W''\cup Z$, Grassmann's formula gives $\dim (\langle \nu 
_d(W''\cup Z)\rangle )= \dim (\langle \nu _d(W'')\rangle) +
\dim (\langle
\nu _d(Z)\rangle ) -\dim (\langle \nu _d(W'')\rangle \cap
\langle \nu _d(Z)\rangle )$. Thus $\dim (\langle \nu _d(Z)\rangle) = 
\dim (\langle \nu _d(W'')\rangle \cap
\langle \nu _d(Z)\rangle )+\deg (Z_4)$. Since $\nu _d(Z)$ is linearly 
independent and $Z = (Z\cap W'')\sqcup Z_4$, we get $\dim (\langle 
\nu _d(Z)\rangle) =
\dim (\langle \nu _d(Z\cap W'')\rangle )+\deg (Z_4)$. Thus $\dim 
(\langle \nu _d(W'')\rangle \cap
\langle \nu _d(Z)\rangle) = \dim (\langle \nu _d(Z\cap W'')\rangle)$. 
Since $\nu_d(W''\cap Z)\subseteq \langle \nu_d (W'')\cap \nu_d(Z) 
\rangle$, $\deg(\langle \nu _d(W'')\rangle \cap
\langle \nu _d(Z)\rangle)=\dim (\langle \nu _d(W''\cap Z) 
\rangle)+1$, and $\nu_d(W'')$ is linearly independent, then the 
linear space $\langle \nu _d(W'')\rangle \cap
\langle \nu _d(Z)\rangle$ is spanned by $\nu _d(W''\cap Z)$. Since $S 
\subseteq W''$ and $P\in \langle \nu _d(Z)\rangle \cap
\langle \nu _d(S)\rangle$, we have $P\in \langle \nu _d(W''\cap 
Z)\rangle$. Since $\nu_d(Z)$ computes $\sbr(P)$, we get
$W''\cap Z = Z$, i.e.
$Z_4=\emptyset$.\\

\quad (c) Here we prove that each point of $S_2$ is a connected 
component of $Z$. Fix $Q\in S_2$ and call $Z_Q$ the connected 
component
of $Z$ such that $(Z_Q)_{red} = \{Q\}$. Set $Z[Q]: = (Z\setminus 
Z_Q)\cup \{Q\}$ and $W_Q:= (W\setminus Z_Q)\cup \{Q\}$. Since
$Z_Q$ is a connected component of
$W$, the schemes $Z[Q]$ and
$W_Q$ are well-defined. Assume $Z_Q\ne \{Q\}$, i.e. $W_Q\ne W$, i.e. 
$Z[Q] \ne Z$. Since $W_Q\cap
\ell = W\cap \ell$, Lemma 2 gives
$\dim (\langle \nu _d(W)\rangle ) - \dim (\langle \nu _d(W_Q)\rangle 
)= \deg (Z_Q)-1>0$. Since $\nu _d(Z)$ is linearly independent, we 
have $\dim (\langle \nu _d(Z)
\rangle) = \dim (\langle \nu _d(Z[Q])\rangle) + \deg (Z_Q)-1$. 
Grassmann's formula gives $\dim (\langle \nu
_d(Z[Q])\rangle )= \dim (\langle \nu _d(W_Q)\rangle \cap \langle \nu 
_d(Z)\rangle )$. Since $\langle \nu _d(Z[Q])\rangle
\subseteq \langle \nu _d(W_Q)\rangle \cap \langle \nu _d(Z)\rangle$ 
and $Z[Q]$ is linearly independent, we get
$\langle \nu _d(Z[Q])\rangle = \langle \nu _d(W_Q)\rangle \cap 
\langle \nu _d(Z)\rangle$. Since $Q\in S_2\subseteq S$, we have
$S \subset W_Q$. Thus $P\in \langle \nu_d(W_Q)\rangle$. Thus $P \in \langle 
\nu _d(Z)\rangle \cap \langle \nu _d(W_Q)\rangle =
\langle \nu _d(Z[Q])\rangle$. Since $\mathcal {Z}$ computes $\sbr(P)$, $Z[Q]\subseteq Z$ and $P\in \langle  \nu
_d(Z[Q])\rangle$, we get $Z[Q] = Z$. Thus each point of $\mathcal 
{S}_2$ is a connected component of $\mathcal {Z}$.\\

\quad (d) To conclude that $Z =( Z\cap \ell )\sqcup S_2$ it is 
sufficient to prove that every connected component of $Z$ whose 
support is a point of $\ell$ is contained in $\ell$. Set $\eta := \deg 
(Z\cap \ell)$
and call $\mu$ the sum of the degrees of the connected components of 
$Z$ whose support is contained in $\ell$. \\
Set $W_1:= (W\cap \ell)\cup S_2$. Notice that $\deg (W_1) = \deg (W) 
+\eta -\mu$. Lemma 2 gives
$\dim (\langle \nu _d(W_1)\rangle ) = \dim  (\langle \nu _d(W)\rangle 
)+\eta -\mu$.
Since $W = W_1\cup Z$, Grassmann's formula gives $\dim (\langle \nu 
_d(W_1\cup Z)\rangle )= \dim (\langle \nu _d(W_1)\rangle ) + \dim 
(\langle \nu _d(Z)\rangle ) -\dim (\langle \nu _d(W_1)\rangle \cap
\langle \nu _d(Z)\rangle )$. Thus $\dim (\langle \nu _d(Z)\rangle) = 
\dim (\langle \nu _d(W_1)\rangle \cap
\langle \nu _d(Z)\rangle )+\mu -\eta$. Notice that $Z\cap W_1 = 
(Z\cap \ell) \sqcup S_2$, i.e. $\deg (Z\cap W_1) = \deg (Z) -\eta +\mu$. 
Since $\nu _d(Z)$ is linearly
independent, we get $\dim (\langle \nu _d(Z)\rangle) = \dim (\langle 
\nu _d(Z\cap W_1)\rangle ) +\mu -\eta$. Thus $\dim (\langle \nu 
_d(W_1)\rangle \cap
\langle \nu _d(Z)\rangle) = \dim (\langle \nu _d(Z\cap W_1)\rangle 
)$, i.e. $\langle \nu _d(W_1)\rangle \cap
\langle \nu _d(Z)\rangle$ is spanned by $\nu _d(W_1\cap Z)$. Since $S 
\subset W_1$ and $P\in \langle \nu _d(Z)\rangle \cap \langle \nu 
_d(S)\rangle$, we
have $P\in \langle \nu _d(W_1\cap Z)\rangle$. Since $Z$ computes the 
symmetric  border rank of $P$, we get $W_1\cap Z = Z$, i.e. $\eta = 
\mu$. Together with
steps (b) and (c) we get $Z = (Z\cap \ell) \sqcup S_2$. Thus from steps 
(b), (c) and (d) we get
$\mathcal {Z} = (\mathcal {Z}\cap C_d)\sqcup \mathcal {S}_2$.
\\

\quad (e) Here we prove the uniqueness of the rational normal
curve $C_d$. Notice that $\ell$ and $C_d=\nu _d(\ell)$ are uniquely 
determined by the choice of a pair $(Z,S)$ with $\nu _d(Z)$ computing 
$\sbr (P)$ and $\nu _d(S)$ computing $\sr (P)$. Fix
another pair $(Z',S')$ with $\nu _d(Z')$ computing $\sbr (P)$ and 
$\nu _d(S')$ computing $\sr (P)$. Let $\ell'$ be the line associated to 
$Z'\cup S'$. Assume $\ell'\ne \ell$.
First assume $S' = S$. The part of Theorem \ref{a1} proved before 
gives $Z = Z_1\sqcup S_2$, $Z' = Z'_1\sqcup
S'_2$ and $S = S'_1\sqcup S'_2$ with $Z_1=Z\cap \ell$, $Z'_1 = Z'\cap 
\ell'$, $S_1 = S\cap \ell$  and $S'_1 =S_1\cap \ell'$.
Now $\sbr
(P) = \deg (Z_1) +\sharp (S_2) = \deg (Z'_1)+\sharp (S'_2)$, $\sr (P) 
= \deg (S_1) +\sharp(S_2) = \deg(S'_1) +\sharp(S'_2)$,
$\deg(S_1) > \deg(Z_1)$, $\deg(S_1) +\deg(Z_1) \ge d+2$ and 
$\deg(S'_1) +\deg(Z'_1) \ge
d+2$. Since $\ell'\ne \ell$, at most one of the points of $S_1$ may be
contained in $\ell'$ and at most one of the points of $S'_1$ may be
contained in $\ell$. Thus $\deg(S'_1) -1 \le \sharp(S_2)$ and $\deg(S_1) 
-1 \le \sharp(S'_2)$.
Since $\deg(S_1) + \deg(Z_1) +2(\sharp(S_2)) = \deg(S'_1) +\deg(Z'_1) 
+2(\sharp(S'_2)) \le 2d+1$,
$\deg(S_1) +\deg(Z_1) \ge d+2$ and $\deg(S'_1) +\deg(Z'_1) \ge d+2$, 
we get $2(\sharp(S_2))
\le d-1$ and $2(\sharp(S'_2)) \le d-1$. Since $\deg(S_1) +\deg(Z_1) 
\ge d+2$ and
$\deg(S_1) >\deg(Z_1)$, we have $\deg(S_1) \ge (d+3)/2$. Hence $\deg(S_1) -1
\ge (d+1)/2 >(d-1)/2 \ge \sharp (S'_2)$, contradiction. Thus all 
pairs $(Z',S)$ give the same line $\ell$. Now assume $S' \ne S$. Call 
$\ell''$ the line associated
to the pair $(Z,S')$. The part of Theorem \ref{a1} proved in the 
previous steps gives that $\ell$ is the only line containing an 
unreduced connected component of $Z$.
Thus $\ell'' =\ell$. Since we proved that the lines associated to $(Z',S')$ 
and $(Z,S')$ are the same, we are done.\\

\quad (f) Here we prove the uniqueness of $\mathcal{S}_2$. Take any 
pair $(Z',S')$ with $\nu _d(Z')$ computing $\sbr (P)$ and $\nu 
_d(S')$ computing $\sr (P)$.  By step (e) the same line $\ell$
is associated to any pair $(Z'',S'')$ as above. Hence the set $S'_2 
:= S'\setminus (S'\cap \ell)$ associated to the pair $(Z,S')$  is the 
union of the connected components of $Z$ not contained in $\ell$. Thus 
$S'_2= S\setminus S\cap \ell
= S_2$.
We apply the part of Theorem \ref{a1} proved in steps (a), (b), (c) 
and (d) to the pair $(Z',S)$. We get that $S\setminus S\cap \ell$ is the 
union of the connected components
of $Z'$ not contained in $\ell$. Applying the same part of Theorem 
\ref{a1} to the pair $(Z',S')$ we get $S'\setminus S'\cap \ell = 
S\setminus S\cap \ell$, concluding the proof of the uniqueness
of $\mathcal {S}_2$. \qed
\\
\\
The following example shows that the assumption 
``~$\sbr(P)+\sr(P)\leq 2d+1$~'' in Theorem \ref{a1} is sharp.

\begin{example}\label{00}
Fix integers $m \ge 2$ and $d \ge 4$. Let $C\subset \mathbb {P}^m$ be 
a smooth conic. Let $Z\subset C$ be any unreduced degree $3$
subscheme. Set $\mathcal {Z}:= \nu _d(Z)$. Since $d\ge 2$, then $\mathcal 
{Z}$ is linearly independent. Since $\mathcal {Z}$ is curvilinear, it 
has
only finitely many degree $2$ subschemes. Thus
the plane $\langle \mathcal {Z}\rangle$ contains only finitely many 
lines spanned by a degree $2$ subscheme of $\mathcal {Z}$. Fix any
$P\in \langle \mathcal {Z}\rangle$ not contained in one of these 
lines. Remark \ref{Z} gives $\sbr (P) =3$. The proof of \cite{bgi}, 
Theorem 4, gives
$\sr (P) = 2d-1$ and the existence of a set $S\subset C$ such that 
$\sharp (S)=2d-1$ , $S\cap Z = \emptyset$ and $\nu _d(S)$ computes 
$\sr (P)$.
We have $\sbr (P) +\sr (P) =2d+2$.
\end{example}

\begin{lemma}\label{dd}
Fix $P\in \PP N$ such that $\rho := \sbr (P) =\sr (P) \le d$. Let 
$\Psi$ be the set of all 0-dimensional
schemes
$A\subset \PP m$  such that $\deg (A) = \rho$ and $P\in \langle \nu 
_d(A)\rangle$. Assume $\sharp (\Psi )\ge 2$. Fix any
$A\in \Psi$. Then there is a line $\ell\subset \PP m$ such that $\deg 
(\ell\cap A)\ge (d+2)/2$.
\end{lemma}

\begin{proof}
Since $\sr (P) =\rho$ and $\sharp (\Psi )\ge 2$, there is $B\in \Psi$ 
such that $B \ne A$ and at least
one among the schemes $A$ and $B$ is reduced.
Since $\deg (A\cup B)\le 2d+1$ and $h^1(\PP m,\mathcal {I}_{A\cup 
B}(d)) >0$, there is a line $\ell\subset \PP m$ such
that $\deg ((A\cup B)\cap \ell) \ge d+2$. We may
repeat verbatim the proof of Theorem \ref{a1}, because it does not 
use the inequality $\deg (A) < \deg (B)$, but only that $\deg (\mathcal {Z}) \le \deg (\mathcal {S})$
and $\mathcal {Z} \ne \mathcal {S}$ (if $\mathcal {T} \ne \mathcal {Z}$, then $\deg (\mathcal {T})<
\deg (\mathcal {Z})\le \deg (\mathcal {S})$ and hence $\mathcal {T} \ne \mathcal {S}$).
We get $A = A_1\sqcup A_2$ and $B=B_1\sqcup A_2$ with $A_2$ reduced, 
$A_2\cap \ell=\emptyset$ and
$A_1\cup B_1\subset \ell$. Since $\deg (A)=\deg (B)$, we have $\deg 
(A_1)=\deg (B_1)$. Thus $\deg (A_1)\ge (d+2)/2$.
\end{proof}

\vspace{0.3cm}

\qquad {\emph {Proof of Theorem \ref{oo}.}} Since $\sbr (P) \le \rho 
\le d$, the border rank is the minimal degree of a smoothable 0-dimensional
scheme $\mathcal {A}\subset X_{m,d}$ such that $P\in \langle \mathcal {A}\rangle$ (Remark 
\ref{Z}). Thus it is sufficient to prove
the last assertion. Assume the existence of a 0-dimensional scheme 
$\mathcal {Z} \subset X_{m,d}$ such that
$z:= \deg (\mathcal {Z})\le \rho$ and $P\in \langle \mathcal 
{Z}\rangle$. If $z=\rho$ we also assume $\mathcal {Z}\ne \nu _d(B)$. 
Taking $z$
minimal, we may also assume $z\le \sbr (P)$. Let $Z\subset \PP m$ be the 
only scheme such that $\nu _d(Z) =\mathcal {Z}$. If $z < \rho$ we 
apply a small part of the proof of Theorem \ref{a1} to the pair 
$(\mathcal {Z},\nu _d(B))$
(we just use or reprove that $\deg ((Z\cup B)\cap \ell) \ge d+2$ and 
that $\deg (B\cap \ell) =\deg (Z\cap \ell) +\rho -z \ge \deg (Z\cap \ell)$). 
We get a contradiction: indeed
$B\cap \ell$ must have degree $\ge (d+1)/2$, contradiction. If $z= 
\rho$, then we use Lemma \ref{dd}.\qed

\begin{example}\label{o+}
Assume $m=2$ and $d \ge 4$. Let $C\subset \PP 2$ be a smooth conic. 
Fix sets $S, S'\subset C$ such that $\sharp (S)=\sharp
(S')=d+1$ and $S\cap S'=\emptyset$. Since no $3$ points of $C$ are 
collinear, the sets $S$, $S'$ and $S\cup S'$ are in linearly
general position. Since
$h^0(C,\mathcal {O}_C(d))=2d+1$ and
$C$ is projectively normal, we have
$h^1(\PP 2,\mathcal {I}_S(d))=h^1(\PP 2,\mathcal {I}_{S'}(d))=0$ and 
$h^1(\PP 2,\mathcal {I}_{S\cup S'}(d))=1$. Thus $\nu
_d(S)$ and $\nu _d(S')$ are linearly independent and $\langle \nu 
_d(S)\rangle \cap \langle \nu _d(S')\rangle$ is a unique
point. Call $P$ this point. Obviously $\sr (P)\le d+1$. In order to get the 
example claimed in the Introduction after the statement of Theorem 
\ref{oo}, it is sufficient to
prove that $\sbr (P) \ge d+1$. Assume $\sbr (P)\le d$ and take $Z$ 
computing $\sbr (P)$. We may apply a small part of the proof of 
Theorem
\ref{a1} to
$P, S, Z$ (even if a priori $S$ may not compute $\sr (P)$). We get 
the existence of a line $\ell$ such
that $\deg (Z\cap \ell) < \sharp (S\cap \ell)$ and $\deg (Z\cap \ell)+\sharp 
(S\cap \ell) \ge d+2$. Since $d \ge 4$, we get
$\sharp (S\cap \ell) \ge 3$, that is a contradiction.
\end{example}

We do not have experimental evidence to raise the
following question (see \cite{bgi} for the cases with $\sbr(P) \le 3$).

\begin{question}\label{q1}
Is it true that $\sr(P) \le d(\sbr(P)-1)$ for all
$P\in \PP N$ and that equality holds
if and only if $P\in TX_{m,d}\setminus X_{m,d}$ where 
$TX_{m,d}\subset \PP N$ is the tangential variety of the Veronese 
variety $X_{m,d}$?
\end{question}

\providecommand{\bysame}{\leavevmode\hbox to3em{\hrulefill}\thinspace}


\begin{thebibliography}{99}

\bibitem{ah} J. Alexander, A. Hirschowitz.
Polynomial interpolation in several variables.
J. Algebraic Geom. 4 (1995), no. 2, 201--222.

\bibitem{bb} E. Ballico, A. Bernardi. Stratification of the fourth 
secant variety of Veronese variety via the symmetric rank. 
arXiv.org/abs/1005.3465~[math.AG].

\bibitem{bgi} A. Bernardi, A. Gimigliano, M. Id\`{a}.
Computing symmetric rank for symmetric tensors.
J. Symbolic. Comput. 46 (2011), 34--55.

\bibitem{bcmt} J. Brachat, P. Comon, B. Mourrain, E. P. Tsigaridas. 
Symmetric tensor decomposition.
Linear Algebra Appl. 433 (2010), no. 11--12, 1851--1872.

\bibitem{bo} M. C. Brambilla, G. Ottaviani. On the 
Alexander-Hirschowitz theorem. J. Pure Appl. Algebra 212 (2008), no. 
5, 1229--1251.

\bibitem{bgl} J. Buczy\'{n}ski, A. Ginensky, J. M. Landsberg. 
Determinantal equations for secant varieties and the 
Eisenbud-Koh-Stillman
conjecture. arXiv:1007.0192~[math.AG].


\bibitem{cc} L. Chiantini, C. Ciliberto. Weakly defective varieties. 
Trans. Amer. Math. Soc. 454 (2002), no. 1, 151--178.

\bibitem{ci} C. Ciliberto.
Geometric aspects of polynomial interpolation in more variables and
of Waring's problem.
European Congress of Mathematics, Vol. I (Barcelona, 2000),
Progr. Math.,  201, Birkh\"auser, Basel, 2001, 289--316.

\bibitem{cmr} C. Ciliberto, M. Mella, F. Russo. Varieties with one 
apparent double point. J. Algebraic Geom. 13 (2004), no. 3, 475--512.

\bibitem{cr} C. Ciliberto, F. Russo. Varieties with minimal secant 
degree and linear systems of maximal dimension on surfaces. Adv. 
Math. 206 (2006), no. 1, 1--50.

\bibitem{cs} G. Comas and M. Seiguer.
On the rank of a binary form. Found. Comp. Math. 11 (2011), no. 1, 65--78.

\bibitem{cglm} P. Comon, G. H. Golub, L.-H. Lim, B. Mourrain.
Symmetric tensors and symmetric tensor rank.
SIAM J.  Matrix Anal.  30 (2008) 1254--1279.

\bibitem{cm} P. Comon, B. Mourrain.
Decomposition of quantics in sums of powers of linear forms.
Signal Processing,
Elsevier 53, 2, 1996.

\bibitem{ik} A. Iarrobino, V. Kanev.
Power sums, Gorenstein algebras, and determinantal loci.
Lecture Notes in
Mathematics, vol. 1721, Springer-Verlag, Berlin, 1999, Appendix C by 
Iarrobino and Steven L. Kleiman.


\bibitem{ls} L.-H. Lim, V. De Silva.
Tensor rank and the ill-posedness of the best low-rank approximation problem.
SIAM J. Matrix Anal. 30 (2008), no. 3, 1084--1127.

\bibitem{lt} J. M. Landsberg, Z. Teitler. On the ranks and border 
ranks of symmetric tensors. Found. Comput. Math. 10, (2010) no. 3, 339--366.


\bibitem{me} M. Mella.
Singularities of linear systems and the Waring problem.
Trans.
Amer. Math. Soc. 358 (2006),  no. 12, 5523--5538.

\bibitem{me1} M. Mella. Base loci of linear systems and the Waring 
problem.  Proc. Amer. Math. Soc.  137  (2009),  no. 1, 91--98.


\bibitem{rs} K. Ranestad and F. O. Schreyer. Varieties of sums of powers.
J. Reine Angew. Math. 525 (2000), 147--181.



\bibitem{vw} R. C. Vaughan, T. D. Wooley.
Waring's problem: a survey, Number theory for the millennium.
III (Urbana, IL, 2000), A K Peters, Natick, MA, (2002),  301--340.


\end{thebibliography}
\end{document}